\documentclass[11pt]{amsart}
\usepackage{amsmath,amssymb,amsthm,amscd,verbatim}
\usepackage{graphicx}
\usepackage{epsfig}
\usepackage{hyperref}
\usepackage{xcolor}

\newtheorem{thm}{Theorem}[section]
\newtheorem{lem}[thm]{Lemma}

\newtheorem{prop}[thm]{Proposition}
\newtheorem{cor}[thm]{Corollary}
\newtheorem{obs}[thm]{Observation}

\newtheorem*{thm*}{Theorem}

\theoremstyle{definition}

\newtheorem{qn}[thm]{Question}

\theoremstyle{remark}

\newtheorem*{acknowledgement}{Acknowledgments}

\newcommand{\Lo}{\textsf{LOCAL}}

\newcommand{\mad}{\ensuremath{\mathrm{mad}}}
\newcommand{\ch}{\ensuremath{\mathrm{ch}}}

\newcommand{\plog}{\ensuremath{\mathrm{polylog}}}

\renewcommand{\le}{\leqslant}
\renewcommand{\leq}{\leqslant}
\renewcommand{\ge}{\geqslant}

\def\longequation{$$\vcenter\bgroup\advance\hsize by -9em%
\noindent\ignorespaces\refstepcounter{equation}}%
\makeatletter%
\def\endlongequation{\egroup\eqno(\theequation)$$\global\@ignoretrue}
\makeatother

\begin{document}
\title{Distributed coloring in sparse graphs with
  fewer colors}

\author{Pierre Aboulker} \address{Laboratoire G-SCOP (CNRS,
  Univ. Grenoble Alpes), Grenoble, France}
\email{pierre.aboulker@grenoble-inp.fr}

\author{Marthe Bonamy} \address{LaBRI (CNRS,
  Universit\'e de Bordeaux), Bordeaux, France}
\email{marthe.bonamy@u-bordeaux.fr}

\author{Nicolas Bousquet} \address{Laboratoire G-SCOP (CNRS,
  Univ. Grenoble Alpes), Grenoble, France}
\email{nicolas.bousquet@grenoble-inp.fr}

\author{Louis Esperet} \address{Laboratoire G-SCOP (CNRS, Univ. Grenoble Alpes), Grenoble, France}
\email{louis.esperet@grenoble-inp.fr}

\thanks{An extended abstract of this work  appeared in the ACM
  Symposium on Principles of Distributed Computing (PODC), 2018.\newline
Partially supported by ANR Projects STINT
  (\textsc{anr-13-bs02-0007}) and GATO (\textsc{anr-16-ce40-0009-01}), and LabEx PERSYVAL-Lab
  (\textsc{anr-11-labx-0025}).}

\date{}
\sloppy

\begin{abstract}
This paper is concerned with efficiently coloring sparse graphs in the distributed
setting with as few colors as possible. According to the celebrated
Four Color Theorem, planar graphs can be colored with at most 4
colors, and the proof gives a (sequential) quadratic algorithm finding
such a coloring. A natural problem is to improve this complexity in
the distributed setting. Using the fact that planar graphs contain linearly
many vertices of degree at most 6, Goldberg, Plotkin, and Shannon
obtained a deterministic distributed algorithm coloring $n$-vertex
planar graphs with 7 colors in $O(\log n)$ rounds. 
Here, we show how
to color planar graphs with 6 colors in $\plog(n)$ rounds. Our
algorithm indeed works more generally in the list-coloring setting and
for sparse graphs (for such graphs we improve by at least one the number of colors
resulting from an
efficient algorithm of Barenboim and Elkin, at the expense of a slightly worst complexity). Our bounds on the
number of colors turn out to be quite sharp in general. Among other results, we show
that no distributed algorithm can color every $n$-vertex
planar graph with 4 colors in $o(n)$ rounds.
\end{abstract}
\maketitle

\section{Introduction}

\subsection{Coloring sparse graphs}

This paper is devoted to the graph coloring problem in the distributed
model of computation. Graph coloring plays a central role in
distributed algorithms, see the recent survey book of Barenboim and
Elkin~\cite{BE13} for more details and further references. Most
of the research so far has focused on obtaining fast algorithms for
coloring graphs of maximum degree $\Delta$ with $\Delta+1$ colors, or
to allow more colors in order to obtain more efficient algorithms. Our
approach here is quite the opposite. Instead, we are interested in proving ``best possible'' results (in
terms of the number of colors), in a reasonable (say polylogarithmic)
round complexity. By ``best possible'', we mean results that match
the best known existential bounds or the best known bounds following
from efficient sequential algorithms. A
typical example is the case of planar graphs. The famous Four Color
Theorem ensures that these graphs are 4-colorable (and the proof actually yields
a quadratic algorithm), but coloring them using so few colors with an efficient
distributed algorithm has remained elusive. Goldberg, Plotkin, and
Shannon~\cite{GPS88} (see also~\cite{BE10})
obtained a deterministic distributed algorithm coloring $n$-vertex
planar graphs with 7 colors in $O(\log n)$ rounds, but it was not
known\footnote{In~\cite{BE10}, it is mentioned that a parallel
  algorithm of~\cite{GPS88} that 5-colors plane graphs (embedded planar
graphs) can be extended to the distributed setting, but this does not
seem to be correct, since the algorithm relies on edge-contractions and some
clusters might correspond to connected subgraphs of diameter linear
in the order of the original graph. The authors of~\cite{BE10}
acknowledged (private communication) that consequently, the problem of coloring planar graphs with 6-colors in polylogarithmic time was still open.} whether a polylogarithmic 6-coloring algorithm exists for
planar graphs.

\medskip

In this paper we give a simple deterministic distributed 6-coloring
algorithm for planar graphs, of round complexity $O(\log^3 n)$. In
fact, our algorithm works in the more general list-coloring setting, where each
vertex has its own list of $k$ colors (not necessarily integers from 1
to $k$). The algorithm also works more generally for sparse graphs. Here, we
consider the maximum average degree of a graph (see below for precise
definitions) as a sparseness measure. It seems to be better suited for
coloring problems than arboricity, which had been previously
considered~\cite{BE10,GL17}.



\medskip

To state our result more precisely, we start with some 
definitions and classic results on graph coloring.

\subsection{Definitions}

A \emph{coloring} of a graph $G$ is an assignment of colors to the vertices of $G$ so that adjacent vertices are
assigned different colors. The \emph{chromatic number} of $G$, denoted by
$\chi(G)$, is the minimum integer $k$ so that $G$ has a
coloring using colors from $1,\ldots,k$. In this paper it will also be convenient to consider the
following variant of graph coloring. A family of lists
$(L(v))_{v\in G}$ is said to be a \emph{$k$-list-assignment}
if $|L(v)|\ge k$ for every vertex $v\in G$. Given such a
list-assignment $L$, we say that $G$ is
\emph{$L$-list-colorable} if $G$ has a coloring $c$ such that for every
vertex $v\in G$, $c(v)\in L(v)$. We also say that $G$ is
\emph{$k$-list-colorable}, if for any
$k$-list-assignement $L$, the graph $G$ is $L$-list-colorable. By a
slight abuse of notation, we will sometimes write that our algorithm
\emph{finds a $k$-list-coloring of $G$}. This should be understood as:
\emph{for any $k$-list-assignment $L$, our algorithm finds an
  $L$-list-coloring of $G$}. 

The \emph{choice number} of $G$,
denoted by $\ch(G)$, is the minimum integer $k$ such that $G$ is $k$-list-colorable. Note that if all
lists $L(v)$ are equal, then $G$ is $L$-list-colorable if and only if
it is $|L(v)|$-colorable, and thus $\chi(G)\le \ch(G)$ for any graph
$G$. On the other hand, it is well known that complete bipartite
graphs (and more generally graphs with arbitrary large average degree)
have arbitrary large choice number.

\medskip

The \emph{average degree} of a graph $G=(V,E)$ is defined as the average of the
degrees of the vertices of $G$ (it is equal to 0 if $V$ is empty and to $2|E|/|V|$ otherwise). The \emph{maximum average degree} of a graph $G$, denoted by
$\mad(G)$, is the maximum of the average degrees of the
subgraphs of $G$. The maximum average degree is a
standard measure of the sparseness of a graph. Note that if a graph
$G$ has $\mad(G)<k$, for some integer $k$, then any subgraph of $G$
contains a vertex of degree at most $k-1$, and in particular a simple
greedy algorithm shows that $G$ has (list)-chromatic number at most
$k$. Therefore, for any graph $G$, $\chi(G)\le \ch(G)\le \lfloor
\mad(G)\rfloor+1$. This bound can be slightly improved when $G$ does
not contain a simple obstruction (a large clique), as will be explained below.

\subsection{Previous results}\label{sec:prevr}

Most of the research on distributed coloring of sparse graphs so
far~\cite{BE10,GL17} has
focused on a different sparseness parameter: The \emph{arboricity} of a
graph $G$, denoted by $a(G)$, is the minimum number of edge-disjoint
forests into which the edges of $G$ can be partitioned. By a classic
theorem of Nash-Williams~\cite{NW64}, we have

$$a(G)=\max \left\{ \left\lceil \tfrac{|E(H)|}{|V(H)|-1}\right\rceil \, | \, H \subseteq G,
  |V(H)|\ge 2\right\}.$$

From this result, it is not difficult to show that for any graph $G$,
$2a(G)-2\le \lceil \mad(G) \rceil\le 2a(G)$ (the lower bound is
attained for graphs whose maximum average degree is an even integer).

\medskip

In~\cite{BE10}, Barenboim and Elkin gave, for any $\epsilon>0$, a
deterministic distributed algorithm coloring $n$-vertex graphs of
arboricity $a$ with $\lfloor (2+\epsilon)a\rfloor+1$ colors in
$O(\tfrac{a}{\epsilon} \log n)$ rounds. In particular, their
algorithm colors $n$-vertex graphs of
arboricity $a$ with $2a+1$ colors in
$O(a^2 \log n)$ rounds.

However it is not difficult to prove that graphs with
arboricity $a$ are $(2a-1)$-degenerate (meaning that every subgraph contains a
vertex of degree at most $2a-1$), and thus $2a$-colorable, which is sharp. A natural
question is whether there is a fundamental barrier for obtaining an
efficient distributed algorithm coloring graphs of arboricity $a$ with
$2a$ colors. It turns out that there is such a barrier when $a=1$,
i.e.\ when $G$ is a tree: It was proved by Linial~\cite{L92} that
coloring a path (and thus a tree) with two colors requires a linear
number of rounds. Yet, our main result will easily imply that the
case $a=1$ is an exception: when
$a\ge 2$, there is a fairly simple distributed algorithm running in
$O(a^4 \log^3 n)$ rounds, that colors graphs of arboricity $a$ with
$2a$ colors.

\subsection{Brooks theorem, Gallai trees, and list-coloring}

A classic theorem of Brooks states that any connected graph of maximum degree
$\Delta$ which is not an odd cycle or a clique has chromatic number at
most $\Delta$. This improves the simple bound of $\Delta+1$ obtained
from the greedy coloring algorithm. While most of the research in
coloring in the distributed computing setting has focused on
$(\Delta+1)$-coloring, Panconesi and Srinivasan~\cite{PS95} gave a
$O(\Delta \log^3 n/\log \Delta)$ deterministic distributed algorithm that given a connected graph $G$ of maximum
degree $\Delta\ge 3$ finds a clique $K_{\Delta+1}$ or a
$\Delta$-coloring. In Section~\ref{sec:cons} we show how to simply
derive a list-version of this result from our main result, at the cost
of an increased dependence in $\Delta$ (see Corollary~\ref{cor:brooks}).

\medskip

\begin{figure}[htbp]
\begin{center}
\includegraphics[width=6cm]{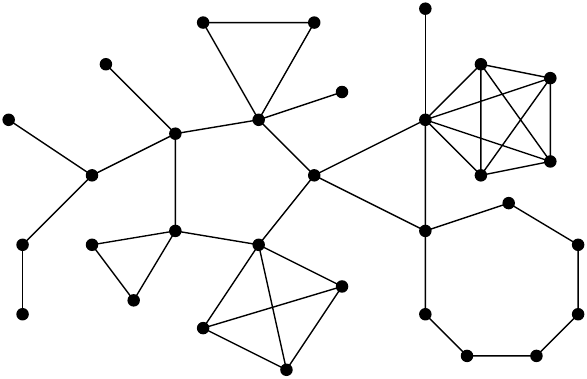}
\caption{A Gallai tree.}
\label{fig:gtree}
\end{center}
\end{figure}

A \emph{block} of a graph $G$ is a maximal 2-connected subgraph of $G$.
A \emph{Gallai tree} is a connected graph in which each block
is an odd cycle or a clique, see Figure~\ref{fig:gtree} for an example. Note that a
tree is also a Gallai tree, since each block of a tree is an edge
(i.e. a
clique on two vertices). The degree of a vertex $v$ in a graph $G$ is
denoted by $d_G(v)$.
The proof of our main result is mainly based on the following classic theorem in graph theory proved independently
by Borodin~\cite{Bor77} and Erd\H{o}s, Rubin, and Taylor~\cite{ERT79},
extending Brooks theorem (mentioned above) to the list-coloring
setting. 

\begin{thm}[\cite{Bor77,ERT79}]\label{thm:Gallai}
If a connected graph $G$ is not a Gallai tree, then for any
list-assignment $L$ such that for every vertex $v\in G$, $|L(v)|\ge
d_G(v)$, $G$ is $L$-list-colorable.
\end{thm}

It is not difficult to prove that Theorem~\ref{thm:Gallai} implies
Brooks theorem. We mentioned above that for any graph $G$, $\chi(G)\le \ch(G)\le \lfloor
\mad(G)\rfloor+1$. Let us now see how this can be slightly improved using
Theorem~\ref{thm:Gallai} if we
exclude a simple obstruction, in the spirit of Brooks theorem.

\begin{thm}[Folklore]\label{thm:folk}
Let $G$ be a graph and let $d=\lceil\mad(G)\rceil$. If $d\ge 3$ and $G$ does not contain any
$(d+1)$-clique, then $\chi(G)\le ch(G)\le d$.
\end{thm}

\begin{proof}
We prove the result by induction on the number of vertices of $G$. We
can assume that $G$ is connected, since otherwise we can consider each connected
component separately. If
$G$ contains a vertex $v$ of degree at most $d-1$, we remove it, color the
graph by induction, and then choose for $v$ a color that does not
appear on any of its neighbors. Hence, we can assume that $v$ has
minimum degree at least $d$. Since $G$ has average degree at most $d$,
$G$ is $d$-regular. Note that the only $d$-regular Gallai trees (with $d\ge 3$)
are the $(d+1)$-cliques\footnote{This can be checked by considering a
  leaf block $B$ of the Gallai tree: $B$ is either a cycle (in which
case the Gallai tree contains a vertex of degree 2) or a clique. If
this clique contains a cut-vertex $v$, then the degree of $v$ is
larger than the degree of the other vertices of the clique, and thus the
graph is not regular.}, and thus it
follows from Theorem~\ref{thm:Gallai} that $G$ is $d$-list-colorable.
\end{proof}

\subsection{Our results}

Our main result is an efficient algorithmic counterpart of
Theorem~\ref{thm:folk} in the \Lo\ model of computation~\cite{L92}, which is standard in
distributed graph algorithms. Each node of an $n$-vertex graph $G$ has
a unique identifier (an integer between $1 $ and $n$), and can
exchange messages with its neighbors during synchronous
rounds. In the \Lo\ model, there is no bound on the size of the
messages, and nodes have infinite computational power. Initially, each node only knows its own identifier, as well as $n$
(the number of vertices) and sometimes some other parameters: in
Theorem~\ref{thm:mad} below, each node knows its own list of $d$ colors (in the
list-coloring setting), or simply the integer $d$ (if we are merely
interested in coloring the graph with colors from $1$ to $d$ and there
are no lists involved). With this information, each vertex has to
output its own color in a proper coloring of the graph $G$. The
\emph{round complexity} of the algorithm is the number of rounds
it takes for each vertex to choose a color. In the \Lo\
model of computation, the output of each vertex $v$ only depends on
the labelled ball of radius $r$ of $v$, where $r$ is the round
complexity of the algorithm. In particular, in this model any problem on $G$ can be solved in a number of rounds
that is linear in the diameter of $G$, and thus the major problem is to
obtain bounds on the round complexity that are significantly better
than the diameter. The reader is referred to the survey book of Barenboim and
Elkin~\cite{BE13} for more on coloring algorithms in the  \Lo\
model of computation.

\begin{thm}[Main result]\label{thm:mad}
There is a deterministic distributed algorithm that given an $n$-vertex graph $G$,
and an integer $d\ge \max(3,\mad(G))$, either finds a
$(d+1)$-clique in $G$, or finds a $d$-list-coloring of $G$ in
$O(d^4\log^3 n)$ rounds. Moreover, if every vertex has degree at most $d$, then
the algorithm runs in $O(d^2\log^3 n)$ rounds.
\end{thm}


Noting that graphs of arboricity $a$ have maximum average degree at
most $2a$ and no clique on $2a+1$ vertices, we obtain the following
result as an immediate consequence.

\begin{cor}\label{cor:arbo}
There is a deterministic distributed algorithm that given an
$n$-vertex graph $G$ of arboricity $a\ge 2$, finds a $2a$-list-coloring of $G$ in
$O(a^4\log^3 n)$ rounds.
\end{cor}

Before we discuss other consequences of our result, let us first discuss
its tightness. First, Corollary~\ref{cor:arbo} improves the result of Barenboim and Elkin~\cite{BE10} mentioned
above by at least one color in general, and Theorem~\ref{thm:mad}
improves it by at least 3 colors in some cases (for instance for
graphs whose maximum average degree is an even integer), and both
results are best possible in general in terms of the number of
colors (already from an existential point of view). On the other hand,
the round complexity of our algorithm is slightly
worst, but a classic result of Linial~\cite{L92} shows that trees
cannot be colored in $o(\log n)$ rounds with any constant number of
colors, and this implies that even for fixed $d$ or $a$, the round
complexity in Theorem~\ref{thm:mad} and Corollary~\ref{cor:arbo}
cannot be replaced by $o(\log n)$. Second, another classic
result of Linial~\cite{L92} showing that $n$-vertex paths cannot be 2-colored by
a distributed algorithm using $o(n)$ rounds, also shows that we cannot
omit the assumption that $d\ge 3$ in the statement of
Theorem~\ref{thm:mad} and the assumption that $a\ge 2$ in the statement
of Corollary~\ref{cor:arbo}.

\medskip

We also note that using network decompositions~\cite{PS96}, we can
replace the $O(d^4\log^3 n)$ round complexity in Theorem~\ref{thm:mad}
by $d^3 2^{O(\sqrt{\log n})}$, and the $O(d^2\log^3 n)$ round complexity
by $d 2^{O(\sqrt{\log n})}$ (the multiplicative factor of $d$ is
saved similarly as in~\cite{PS95}). These alternative bounds are not very
satisfying, and in most of the applications we have in mind $d$ is a
constant anyway, so we omit the details. It remains interesting to
obtain a bound on the round complexity that is sublinear in $n$ regardless of the value of $d$.

\medskip

In Section~\ref{sec:cons} we explore various consequences of our
result. In particular, we prove that it gives a 6-(list-)coloring algorithm
for planar graphs
in $O(\log^3 n)$ rounds. On the other hand, we show that an efficient
Four Color Theorem cannot be expected in the distributed setting.

\begin{thm}\label{th:lplan}
No distributed algorithm can 4-color 
every $n$-vertex planar graph in $o(n)$ rounds. 
\end{thm}

In Section~\ref{sec:overview}, we give an overview of our algorithm. The proof
of its correctness (and of its round complexity)
follows from Lemmas~\ref{lem:happy}  and~\ref{lem:ext}, that are respectively proved in Sections~\ref{sec:happy}  and~\ref{sec:ext}.

\section{Consequences of our main result}\label{sec:cons}

It was proved by Panconesi and Srinivasan~\cite{PS95} that there
exists a deterministic distributed  algorithm of round complexity $O(\tfrac{\Delta}{\log \Delta}\log^3 n)$ that given
an $n$-vertex graph of maximum degree $\Delta\ge 3$ distinct from a
clique on $\Delta+1$ vertices, finds a $\Delta$-coloring of $G$. Note
that the following list-version of their result can be deduced as a simple corollary of our
main result.

\begin{cor}\label{cor:brooks}
There is a deterministic distributed  algorithm of round complexity
$O(\Delta^2 \log^3 n)$ that given any $n$-vertex graph of
maximum degree $\Delta\ge 3$, and any $\Delta$-list-assignment $L$ for
the vertices of $G$,
either finds an $L$-list-coloring of $G$, or finds that no such
coloring exists.
\end{cor}

We note that Panconesi and Srinivasan~\cite{PS95} 
also gave a randomized variant of their algorithm working in
$O(\tfrac{\log^3 n}{\log \Delta})$ rounds. In our case, it is not clear
whether we can similarly avoid the multiplicative factor polynomial in $\Delta$ in a
randomized version of our algorithm. 

\medskip

In Section~\ref{sec:conclusion}, we will explain how to obtain an
efficient algorithmic version of a variant of Theorem~\ref{thm:Gallai} (see
Theorem~\ref{thm:degchoos}) which also implies
Corollary~\ref{cor:brooks} but is more flexible (in the sense that it
allows vertices to have lists of different sizes).

\medskip

We now turn to consequences of our main result for sparse graphs (whose maximum average degree is independent of $n$). The \emph{girth} of a graph $G$ is the length of a shortest cycle in $G$. A simple consequence of Euler's formula is the following.

\begin{prop}\label{prop:mad}
Every $n$-vertex planar graph of girth at least $g$ has maximum average degree less than $\tfrac{2g}{g-2}$. In particular, planar graphs have maximum average degree less than 6, triangle-free planar graphs have maximum average degree less than 4, and planar graphs of girth at least 6 have maximum average degree less than 3.
\end{prop}

As a direct consequence of our main result, we obtain:

\begin{cor}\label{cor:planar}
There is a deterministic distributed algorithm of round complexity $O(\log^3
n)$ that given an $n$-vertex planar graph $G$, 
\begin{enumerate}
    \item finds a 6-(list-)coloring of $G$;
    \item finds a 4-(list-)coloring of $G$ if $G$ is triangle-free;
    \item finds a 3-(list-)coloring of $G$ if $G$ has girth at least 6. 
\end{enumerate}
\end{cor}

Let us discuss the tightness of our results. We first consider item 2 of Corollary \ref{cor:planar}, which turns out to be sharp in several ways. First, it is known that there exist triangle-free planar graphs that are not 3-list-colorable~\cite{V95}, so our 4-list-coloring algorithm is the best we can hope for (from an existential point of view). However, a classical theorem of Gr\"{o}tzsch~\cite{G59} states that triangle-free planar graphs are 3-colorable, so a natural question is whether a distributed algorithm can find such a 3-coloring efficiently. 

We now show that 3-coloring triangle-free planar graphs requires $\Omega(n)$ rounds, and even for rectangular grids (which are 2-colorable) finding such a 3-coloring requires $\Omega(\sqrt{n})$ rounds, which is significantly worst
than our polylogarithmic complexity.
The proof relies on the following observation, due to
Linial~\cite{L92}, who first applied it to show that regular trees
cannot be colored with a constant number of colors by a distributed
algorithm using $o(\log n)$ rounds.

\begin{obs}[\cite{L92}]\label{obs:linial}
Let $G$ be a graph, and $H$ be a graph with at most $|V(G)|$ vertices, such that each ball of
radius at most $r+1$ in $H$ is isomorphic to some ball of radius at most $r+1$ in
$G$. Then no distributed algorithm can color $G$ with less than $\chi(H)$ colors in at most $r$ rounds.
\end{obs}

\begin{figure}[htbp]
\begin{center}
\includegraphics[width=5cm]{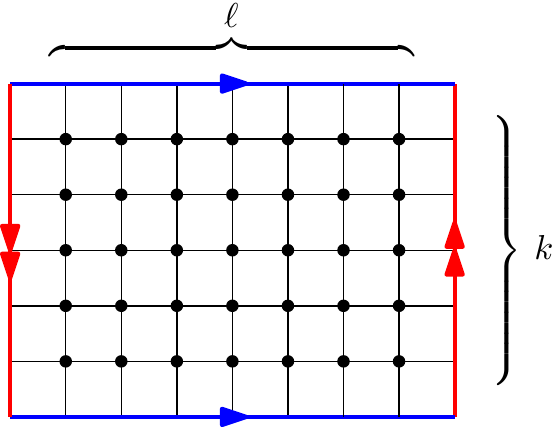} \hspace{2cm}
\includegraphics[width=4cm]{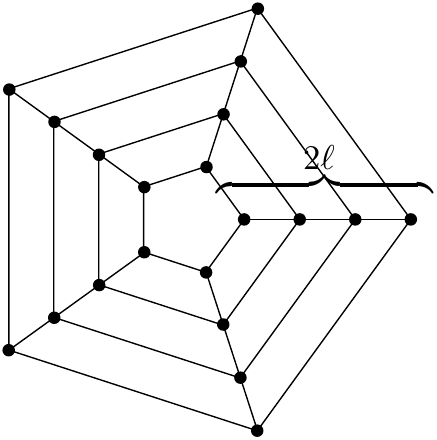}
\caption{Left: A $k$ by $\ell$ grid on the Klein bottle (the graph
  itself can be obtained by identifying the two blue segments (in the
  same direction), and
  the two red segments (in opposite directions)). 
Right: The planar triangle-free graph $H_{2\ell}$ (drawn here for $2\ell=4$).}
\label{fig:grid}
\end{center}
\end{figure}

Let $G_{k,\ell}$ denote the $k$ by $\ell$ rectangular grid on the
Klein bottle (where $k$ denotes the length of vertical cycles and
$\ell$ denotes the length of horizontal cycles, see
Figure~\ref{fig:grid}, left). It was proved by Gallai~\cite{Gal63}
(see also~\cite{KM95,AHNNO01}) that for any $k$ and $\ell$, the graph
$G_{2k+1,2\ell+1}$ is 4-chromatic. Observe that in $G_{2k+1,2k+1}$,
any ball of radius less than $k$ is isomorphic to a ball of the planar
rectangular grid of size $2k+1$ by $2k+1$, and in $G_{5,2\ell+1}$, any ball of radius less than
$\ell$ is isomorphic to a ball of a planar triangle-free graph (more
precisely, a ball in the graph $H_{2\ell}$ depicted in
Figure~\ref{fig:grid}, right). Using Observation~\ref{obs:linial}, this implies the following.

\begin{thm}\label{th:ltfree}
No distributed algorithm can 3-color the graph $H_k$ in less than
$k/2$ rounds. In particular, no distributed algorithm can 3-color
every planar triangle-free graph on $n$ vertices in $o(n)$ rounds. 
\end{thm}

\begin{thm}\label{th:lgrid}
No distributed algorithm can 3-color the rectangular $k\times k$-grid
in the plane  in less than $k/2$ rounds. In particular, no distributed
algorithm can 3-color every planar bipartite graph on $n$ vertices in $o(\sqrt{n})$ rounds. 
\end{thm}

It is an interesting problem to find an algorithm of round complexity
matching this lower bound in the
case of planar bipartite graphs.

\begin{qn}
Is there a distributed algorithm that  can 3-color every $n$-vertex
planar bipartite graph in $O(\sqrt n)$ rounds?
\end{qn}

\begin{figure}[htbp]
\begin{center}
\includegraphics[width=10cm]{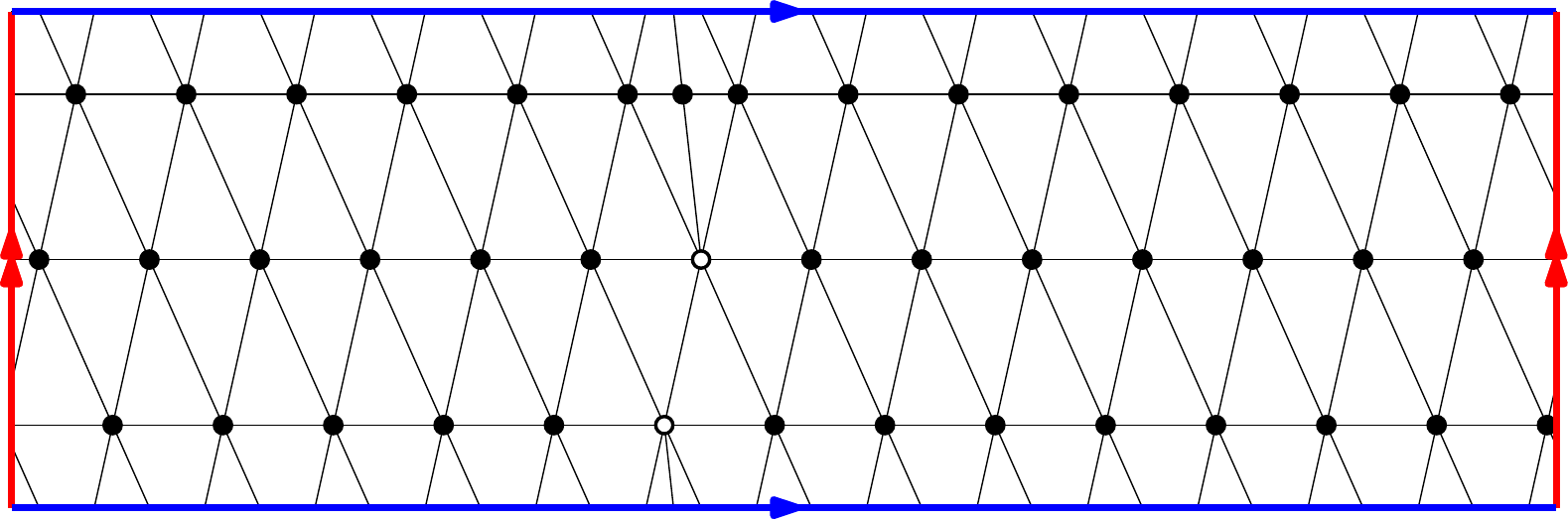} 
\caption{A toroidal triangulation that is not 4-colorable. The graph
  itself is obtained by identifying the two blue segments, and
  the two red segments.}
\label{fig:fisk}
\end{center}
\end{figure}

Consider now item (1) in Corollary \ref{cor:planar}. It is known
that planar graphs are 5-list-colorable~\cite{T94}, so item (1) is not best
possible (from an existential point of view). On the other hand,
Voigt~\cite{V93} proved that there exist planar graphs that are not
4-list-colorable. Fisk~\cite{Fis78} proved that triangulations of 
surfaces in which all vertices have even degree except two adjacent
vertices, are not 4-colorable. Such triangulations exist in the
torus, see Figure~\ref{fig:fisk}, where the two vertices of odd degree
are depicted in white. Versions of this graph on $n$ vertices clearly exist for
any $n\equiv 1 \pmod 3$, and in such a graph any ball of radius at
most $\tfrac{n-1}6-3$ induces a planar graph. Using
Observation~\ref{obs:linial} again, this implies that no distributed
algorithm can 4-color every planar graph on $n$ vertices in $o(n)$
rounds, which proves Theorem~\ref{th:lplan}. This raises the following
natural question.

\begin{qn}
Is it true that there is a distributed algorithm 5-(list-)coloring planar graphs in a polylogarithmic number of rounds?
\end{qn}

It might be the case that the girth condition in item (3) of
Corollary~\ref{cor:planar} is not best possible, so the following might very well have a positive answer.

\begin{qn}
Is it true that there is a distributed algorithm 3-(list-)coloring planar graphs of girth at least 5 in a polylogarithmic number of rounds? 
\end{qn}

By the result of Linial~\cite{L92} (mentioned above), stating that
regular trees cannot be colored with a constant number of colors by a
distributed algorithm within $o(\log n)$ rounds, there is no hope to
obtain a better round complexity for the two questions above.

\medskip

A possible way to show, as in Theorems~\ref{th:ltfree}
and~\ref{th:lgrid}, that planar graphs cannot be efficiently
5-(list-)colored would be to find a graph $G$ embedded on some
surface, in which each ball of sufficiently large radius (say
$n^\epsilon$, for some arbitrary small $\epsilon>0$) is planar, and
such that $G$ is not 5-(list-)colorable. However, such a graph does not exist, as we now explain.

\smallskip

Given a graph $G$ embedded in some surface, the \emph{edge-width} of $G$ is the length of a shortest non-contractible cycle in $G$. The reader is referred to~\cite{MT} for some
background on graphs on surfaces. Note that if $G$ has edge-width at least $2r+2$, then each ball of radius at most $r$ is planar (the converse is not true, as seen for example by considering a 3 by $(2r+2)$ grid on the torus, in which the edge-width is equal to 3, but any $r$-ball is planar). It was proved by Thomassen~\cite{T93} that any graph embedded in some surface of genus $g$ with edge-width at least $2^{O(g)}$ is 5-colorable. DeVos, Kawarabayashi and Mohar~\cite{DKM08} later  proved that embedded graphs of sufficiently large edge-width are 5-list-colorable. These results were qualitatively improved recently by Postle and Thomas~\cite{PT16}, who showed the following.

\begin{thm}[\cite{PT16}]
If $G$ is embedded in a surface of genus $g$, with edge-width $\Omega(\log g)$, then $G$ is 5-list-colorable. Moreover, if $G$ has girth at least 5, then $G$ is 3-list-colorable.
\end{thm}

In their statement, the condition on the edge-width can
be replaced by the weaker condition that every ball of radius $O(\log
n)$ is planar\footnote{Luke Postle, private communication.}. Note that any graph on $n$ vertices has genus at most $O(n^2)$, so it
follows that arguments similar to those of Theorem~\ref{th:ltfree}
and~\ref{th:lgrid} cannot prove that planar graphs on $n$ vertices
cannot be 5-(list-)colored in $O(\log n)$ rounds. And similarly, the
same techniques cannot prove that $n$-vertex graphs of girth at least
5 cannot be $3$-(list\-)colored in in $O(\log n)$ rounds.

\medskip

We conclude this section with a consequence of Theorem~\ref{thm:mad} for graphs
embeddable on a fixed surface (other than the sphere). A classic result of Heawood states that
any graph of Euler genus $g>0$ has maximum average degree at most
$\tfrac{1}{2}(5+\sqrt{24g+1})$, and thus choice number at most
$H(g)=\lfloor\tfrac{1}{2}(7+\sqrt{24g+1})\rfloor$ (see
also~\cite{BMT99} and the references therein). Using this bound, we obtain the following direct corollary of
Theorem~\ref{thm:mad}.

\begin{cor}\label{cor:surf}
For any integer $g\ge 1$, there is a deterministic distributed algorithm of round complexity $O(\log^3
n)$ that given an $n$-vertex graph $G$ embeddable on a surface of Euler
genus $g$, finds an $H(g)$-list-coloring of $G$. Moreover, when
$\tfrac{1}{2}(5+\sqrt{24g+1})$ is an integer and $G$ is not the complete
graph on $H(g)$ vertices, the algorithm can indeed find an $(H(g)-1)$-list-coloring of $G$.
\end{cor}

If we are merely interested in the existence of such a coloring, the assumption that
$\tfrac{1}{2}(5+\sqrt{24g+1})$ is an integer is not necessary in
second part of Corollary~\ref{cor:surf}, as proved by B\"ohme,
Mohar and Stiebitz~\cite{BMT99}. From an algorithmic point of view however, it is not clear
whether this assumption can be omitted.

\section{Overview of the proof of Theorem~\ref{thm:mad}}\label{sec:overview}

We now recall the setting of Theorem~\ref{thm:mad}. The graph $G$ has
$n$ vertices, and maximum average degree at most $d$, for some integer
$d\ge 3$. Moreover, $G$ does not contain any clique on $d+1$ vertices
(otherwise such a clique can be found in two rounds, and we are
done). Any vertex $v$ has a list $L(v)$ of $d$ colors, and our goal is
to efficiently find an $L$-list-coloring.

\medskip

A first remark is that we can assume without loss of generality
that $d\le n$, since otherwise we can simply set $d:=n$ in the
theorem. This will be assumed implicitly throughout the proof (in fact
it will only be needed towards the end of the proof of Proposition~\ref{prop:S1}).

\medskip

The proof of Theorem~\ref{thm:mad} goes as follows: in a
polylogarithmic number of rounds, we identify some set $A$ of vertices of $G$, representing a constant
fraction of the vertex set, and such that any list-coloring of $
G-A$ can be extended to $A$ in a polylogarithmic number of rounds. By
repeatedly removing such a set $A$ (this can be done at most $O(\log
n)$ times), we obtain a trivial graph (that can easily be colored) and then
proceed to extend this coloring to each of the sets $A$
that were removed, one by one (starting from the last one to the first one).

\medskip

So the proof naturally breaks into two very different parts. In the
first one, we find such a set $A$ and show that it has size linear in $n$. We
use purely graph theoretic arguments but we feel that some of our tools
could be useful to design other distributed algorithms. In the second part, we
show how to extend any coloring of $G-A$ to $G$ in a polylogarithmic
number of rounds. The analysis uses a combination of Theorem~\ref{thm:Gallai}
and classic tools from distributed computing.

\medskip

Let us now be more specific about the set $A$ mentioned above. Let
$c = \tfrac{12}{\log(6/5)}$ (this specific value will only be needed in
the proof of Proposition~\ref{prop:S1}, so anywhere else in the proof the
reader can simply assume that $c$ is any fixed constant). Given a vertex $v$ and
an integer $r$, the \emph{ball centered in $v$ of radius $r$}, denoted
by $B^r(v)$, is the set of vertices at distance at most $r$ from
$v$. We will often consider balls of radius $r= \lceil c \log n \rceil$, in this case
we will omit the superscript in $B^{ \lceil c \log n  \rceil}(v)$ and write $B(v)$
instead. Note that from now on we will omit floors and ceiling when they are not
necessary in our discussions (we will write $c \log n$ instead of $\lceil c \log n \rceil$
and consider it as an
integer). We will also be interested in balls within specific
subgraphs of $G$. Given a subset $R$ of vertices of $G$, $B_R^r(v)$ is
the subset of vertices of $G$ that are at distance at most $r$ from
$v$ in $G[R]$ (the subgraph of $G$ induced by $R$). Note that the
ball $B_R^r(v)$ is empty if and only if $v \not\in R$. Again, for
convenience, we write $B_R(v)$ instead of $B_R^{c \log n}(v)$.

\medskip

Any vertex of degree at most $d$ in $G$ is said to be \emph{rich}, and the
remaining vertices are said to be \emph{poor}. Note that there are at
most $\tfrac{d}{d+1}\,n$ poor vertices, and thus the set $R$
of rich vertices has size at least $\tfrac{n}{d+1}$ (a formal proof will be 
provided at the end of the proof of Lemma~\ref{lem:happy}). Our goal is to select a
large set of vertices that can be easily (and efficiently) colored,
given a partial coloring of
the rest of the graph. Vertices $v$ of degree at most $d-1$ certainly
have this property, since their lists contain at least one more color
than their degree (and then the coloring of $G-v$ can be extended). 
However, it might be the case that $G$ contains few
such vertices, or even no such vertex at all (if $G$ is $d$-regular),
and thus we also have to look for candidates in the set of vertices of
degree precisely $d$. Indeed, we will not consider poor vertices (whose degree
is at least $d+1$) as possible candidates. The ball $B_R(v)$
(as defined in the preceding paragraph)
of a rich vertex $v$ is called \emph{the rich ball of $v$}. A rich
vertex is said to be \emph{happy} if its rich ball $B_R(v)$ contains a vertex of
degree at most $d-1$ in $G$ or is not a Gallai tree. Recall that this
second option is equivalent to the fact that some block (2-connected
component) of the subgraph of $G$
induced by $B_R(v)$ is neither an odd cycle nor a clique. We denote by $A$ the set of happy
vertices. In Section~\ref{sec:happy}, we prove the following result.

\begin{lem}\label{lem:happy}
$|A|\ge \tfrac{n}{(3d)^3}$. Moreover, if there are no poor vertices in
$G$, then $|A|\ge \tfrac{n}{12d+1}$.
\end{lem}

Our goal is then to prove that any coloring of $G-A$ can be efficiently
extended to $A$.

\begin{lem}\label{lem:ext}
Any $L$-list-coloring of $G-A$ can be extended to an $L$-list-coloring
of $G$ in $O(d \log^2 n)$ rounds.
\end{lem}

Lemma~\ref{lem:ext} will be proved in Section~\ref{sec:ext}. In the
remainder of this section, we show how to deduce Theorem~\ref{thm:mad}
from Lemmas~\ref{lem:happy} and~\ref{lem:ext}.

\bigskip

\noindent \emph{Proof of Theorem~\ref{thm:mad}.}
Since being happy only
depends on the ball of radius $c\log n$ around each vertex, the set $A$ of happy vertices
can be found in $O(\log n)$ rounds.
We repeatedly remove sets $A_1,A_2,\ldots,A_k$ from $G$ using
Lemma~\ref{lem:happy} until $G$ is empty. By Lemma~\ref{lem:happy},
$k\le \tfrac{\log n}{\log (1/(1-1/27d^3))}=O(d^3\log n)$, and thus this part of the
procedure takes $O(d^3\log^2 n)$ rounds. If each vertex of $G$ has
degree at most $d$, then there are no poor vertices and it follows similarly
from Lemma~\ref{lem:happy} that  $k=O(d\log n)$ and this  part of the
procedure takes $O(d\log^2 n)$ rounds.

We then extend the list-coloring of the empty graph to
$A_k,A_{k-1},\ldots ,A_1$ using Lemma~\ref{lem:ext}. Each extension
takes $O(d \log^2 n)$ rounds, so this part of the procedure runs in
$O(kd \log^2 n)=O(d^4 \log^3 n)$ rounds. If each vertex of $G$ has
degree at most $d$, this part of the procedure runs in
$O(kd \log^2 n)=O(d^2 \log^3 n)$ rounds. In the end, we obtain an
$L$-list-coloring of $G$ in $O(d^4 \log^3 n)$
rounds (or $O(d^2 \log^3 n)$
rounds if $G$ has maximum degree at most $d$), as desired. This concludes the proof of
Theorem~\ref{thm:mad}.\hfill $\Box$

\bigskip

Observe that the proof above basically says that the round complexity
in Theorem~\ref{thm:mad} is at most $O(\log n)$ times $O(d \log^2 n)$ (the complexity of
Lemma~\ref{lem:ext}) times $O(d^3)$ or  $O(d)$ (the denominators in the bounds of
Lemma~\ref{lem:happy}). So any improvement in either of these two lemmas would
yield an immediate improvement in Theorem~\ref{thm:mad}.

\section{$A$ has linear size -- Proof of Lemma~\ref{lem:happy}}\label{sec:happy}

The goal of this section is to prove that $A$, the set of happy
vertices, has size linear in $n$.

\medskip

Recall that a vertex is happy if it is rich (i.e. it has degree at
most $d$) and its rich ball $B_R(v)=B_R^{c\log n}(v)$ contains a vertex of degree at
most $d-1$ or is not a Gallai tree. Let $S=R\setminus A$ (here `S' stands for `Sad'). Note that
$S$ is the set of rich vertices whose rich ball contains only
vertices of degree $d$ (in $G$) and induces a Gallai tree. We will
prove that $S$ is
not too large compared to $A$ (Proposition~\ref{prop:S1} below), and since $R=A\cup S$ has size linear
in $n$, this will prove Lemma~\ref{lem:happy}.


\medskip

Recall that the \emph{girth} of a graph $G$ is the length of a shortest cycle in
$G$. Note that if $G$ has girth at least $g$, then any ball
$B^{\tfrac{g-1}2}(v)$ induces a tree (and thus a Gallai tree).
We will need a consequence of the following result of Alon, Hoory and
Linial~\cite{AHL02}.

\begin{thm}[\cite{AHL02}]\label{thm:moore}
If a graph $G$ has girth at least $g$ ($g$ odd), and average degree
$d=2+\delta$, for some real number $\delta>0$, then $$n\ge 1+
d\sum_{i=0}^{\tfrac{g-1}2}(d-1)^i\ge (1+\delta)^{\tfrac{g-1}2}.$$
\end{thm}

We will only use the following direct corollary of Theorem~\ref{thm:moore}.

\begin{cor}\label{cor:moore}
If an $n$-vertex graph $G$ has girth at least $g$, and average degree at least
$2+\delta$, for some real number $\delta>0$, then $$g\le
\tfrac4{\log(1+\delta)}\,\log n.$$
\end{cor}

\begin{proof}
We can assume that $g$ or $g-1$ is odd, and thus by
Theorem~\ref{thm:moore} we have $n\ge
(1+\delta)^{\tfrac{g}2-1}$. It follows that $$g\le
\tfrac{2 \log n}{\log(1+\delta)}+2\le \tfrac4{\log(1+\delta)}\,\log n,$$ since $1+\delta\le n$.
\end{proof}





Note that any induced subgraph of a Gallai tree is also a Gallai
tree. Thus, in $G[S]$, all balls of radius $c\log n$ are also Gallai
trees. A block of a ball of radius $c\log n$ in $G[S]$ is
called \emph{a local block} of $G[S]$. Observe that if a local block $B$
of $G[S]$ is not a block of $G[S]$, then $B$ is contained in a block
of $G[S]$ that contains a cycle of length greater than $2c \log n$.
Another important observation about local blocks of $G[S]$ is as
follows. 

\begin{obs}\label{obs:maxinclusion}
If three vertices $u,v,w$ of a maximal clique $K$ are in a local block
of $G[S]$, then $K$ is a local block of $G[S]$.
\end{obs}

To see why this holds, observe that in any ball $B(x)$ of $G[S]$ where
$u,v,w$ are in a block $C$ of the ball (by definition the ball $B(x)$ induces a Gallai
tree, and thus this block $C$ is a clique on at least 3 vertices), some vertex of $C$, say $u$, is closer from $x$ than the others
(it is the vertex to which $C$ is attached to the rest of the
Gallai tree induced by $B(x)$). Since all the vertices of $K$ are at
least as close from $x$ as $v$, $K \subseteq B(x)$ and thus $K=C$ itself is a
local block of $G[S]$, which proves
Observation~\ref{obs:maxinclusion}.

\medskip

The bulk of the proof of Lemma~\ref{lem:happy} is contained in the
following technical proposition.

\begin{prop}\label{prop:S1}
There are at least $\tfrac{1}{12}\,
|S|$ vertices of degree at most $d-1$ in
$G[S]$. 
\end{prop}

\begin{proof}
Let $H$ be the graph obtained from $G[S]$ by doing the following: 
\begin{itemize}
 \item First, for each local block $C$
isomorphic to a clique on at least 3 vertices, we add a vertex $v_C$ adjacent to
all the vertices of $C$, and then remove all the edges of
$C$. By definition, all vertices $v_C$  have degree at least 3. These
vertices are drawn in white in Figure~\ref{fig:constrh},
center and right.

Note that we might have decreased the degree of some vertices in the
process: in particular there might be some vertices of degree at least
3 in $G[S]$ that have degree precisely 2 in the current construction. 
Let us call $T$ this set of
vertices (these vertices are circled in Figure~\ref{fig:constrh},
center). Note that each vertex of $T$ has to be adjacent to a vertex $v_C$
(for some clique $C$ on at least 3 vertices, and so $v_C$ has degree
at least 3), and thus there is no
path of three consecutive vertices of $T$ in the current construction. Note that since every ball of
radius $c\log n$ in $G[S]$ is a Gallai tree, vertices of $T$ are not
contained in any cycle of size $2c\log n$.

\item  Then we suppress all the vertices of $T$. Here, by
\emph{suppressing} a vertex of degree two, we mean deleting this vertex, and
then adding an edge between its two neighbors if they are not already
adjacent (in other words, we replace a
path on 2 edges by a path on 1 edge). Note that if we suppress two
adjacent vertices of degree two, this is equivalent to replacing a
path on 3 edges by a path on 1 edge. A direct consequence of the final
remark in the paragraph above is that $H$ has girth at least 5 (in
particular the construction does not produce loops or multiple edges).
\end{itemize}

The two steps of
the construction of $H$ are depicted in Figure~\ref{fig:constrh}.

\begin{figure}[htbp]
\begin{center}
\includegraphics[width=16cm]{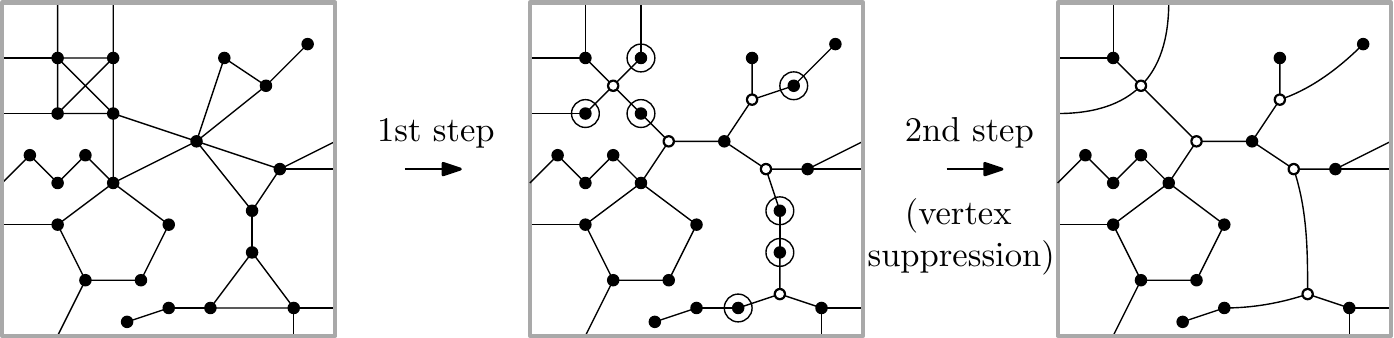}
\caption{The two-step construction of $H$ from $G[S]$ (viewed from a
  ball of radius $c\log n$).}
\label{fig:constrh}
\end{center}
\end{figure}

Observe that if a vertex $v$ of $G[S]$ has degree at least 3 in $G[S]$
and at most 2 in $H$, then it follows from our vertex suppression step that
$v$ has degree at most 1 in $H$. Such a vertex $v$ thus lies in a
unique local block $B$ in $G[S]$, which is a clique on at least 3
vertices. 
Since cliques have size at most $d$ in $G$, such a vertex $v$ has degree
at most $d-1$ in $G[S]$. Since $d\ge 3$, it follows that the number of
vertices of degree at most $d-1$ in
$G[S]$ is at least the number of
vertices of degree at most 2 in
$H$. Hence, to prove the theorem it is enough to show that there are  at least $\tfrac{1}{12}\,
|S|$ vertices of degree at most 2 in
$H$.

\medskip

Before proving this, let us do two convenient observations on the number
  of vertices and the size of cycles in $H$.

\begin{itemize}
\item Since each vertex of $R$ has degree at
most $d$, each vertex of $S$ lies in
at most $d/2$ cliques of at least 3 vertices which are local blocks.
Indeed, (i) an edge cannot be in two distinct maximal cliques by inclusion,
otherwise the ball of radius one centered on one of the vertices of the edge would
not be a Gallai tree, contradicting the fact that it is in $S$. And (ii) local blocks
which are cliques of size at least $3$ are maximal by inclusion by Observation~\ref{obs:maxinclusion}.
Thus $H$ contains at most
$(1+\tfrac{d}6)\,|S|\le \tfrac{d}2\,|S|$ vertices (here we used that
$d\ge 3$). 

\item Consider any cycle $C$ of length $\ell$ in $H$. By construction,
  $C$ is either a local
  block of $G[S]$ (in which case $\ell$ is odd and $5\le \ell \le 2c
  \log n$, since triangles are considered as cliques on 3 vertices and every edge is in at most one local block which is a clique as we already observed in the previous point), or corresponds to a cycle of $G[S]$ of length at least
  $\ell/3$ that is not contained in any ball of radius $c\log n$ of
  $G[S]$. The bound $\ell/3$ follows from the fact that there were no
  paths of 3 vertices of $T$ before the suppression step, and thus we
  have only replaced paths on 2 or 3 edges by single edges. It follows
  that any cycle of $H$ distinct from a local block of $G[S]$ has
  length greater than $\tfrac{2c}3\log n$.
\end{itemize}

\medskip

Let $H'$ be the graph obtained from $H$ by removing precisely one edge
in each cycle of length between 5 and
$\tfrac{2c}3\log n$. Note that the observation above ensures that such a
cycle is necessarily a local block of $G[S]$, and is thus an induced odd cycle.
Such cycles are pairwise edge-disjoint otherwise any vertex incident
to an edge lying in the intersection of two such cycles
would be happy. It follows from the construction that $H'$
has girth greater than $\tfrac{2c}3\log n$.  Since each cycle of $H$ contains at
least 5 edges, we have removed at most
a fifth of the edges of $H$, and thus the average degree of $H$ is at
most $\tfrac54$ times the average degree of $H'$.

Assume for the sake of contradiction that $H'$ has average degree at least $\tfrac{11}{5}$.
Applying
Corollary~\ref{cor:moore} to the graph $H'$ with $g>\tfrac{2c}3\log n$ and
$\delta=\tfrac15$, we obtain $\tfrac{2c}3 \log n < \tfrac4{\log(6/5)}\,\log
|V(H)|\le \tfrac4{\log(6/5)}\,\log
\tfrac{d|S|}2\le \tfrac8{\log(6/5)}\,\log
n$, using that $d\le n$ and $|S|\le n$. This implies that $c < \tfrac{12}{\log(6/5)}$, which contradicts our definition
of $c$.
We can thus assume that $H'$ has average degree at most
$\tfrac{11}{5}$, and thus $H$ has average degree at most $\tfrac54
\cdot \tfrac{11}{5}=\tfrac{11}{4}$. It follows that $H$ has at
least $|S|/12$ vertices of degree at most two, and thus $G[S]$ has at
least $|S|/12$ vertices of degree at most $d-1$, which concludes the
proof of Proposition~\ref{prop:S1}.
\end{proof}

We are now ready to prove Lemma~\ref{lem:happy}.

\bigskip

\noindent \emph{Proof of Lemma~\ref{lem:happy}}. 
By Proposition~\ref{prop:S1}, there are at least $\tfrac{1}{12}\,
|S|$ vertices of degree at most $d-1$ in
$G[S]$. Since each vertex of $S$ has degree $d$ in $G$, this implies
that there are also at least
$\tfrac{1}{12}\,
|S|$ edges leaving $S$. Let us divide these edges into two
sets: $E(S,P)$, the set of edges with one end in $S$ and the other in
$P$, the set of poor vertices, and $E(S,A)$, the set of edges with one end in $S$ and the other
in $A$. Observe that since each vertex of $A$ has degree at most $d$,
we have
$|A|\ge \tfrac1{d}|E(S,A)|$.

\medskip

Since $G$ has average degree at most $d$, we have $\sum_{v \in G}
(d_G(v)-d)\le 0$. Recall that $P$ (the set of poor vertices) is
precisely the set
of vertices of degree at least $d+1$ in $G$, while $A$ (the set of
happy vertices) contains all the vertices of degree at
most $d-1$ in $G$, and thus $$d|A|\ge \sum_{v\in A}(d-d_G(v))\ge
\sum_{v\in P}(d_G(v)-d)\ge \max(|P|,|E(S,P)|-d|P|).$$ It follows that
$|E(S,P)|\le d|P|+d|A|\le d(d+1)|A|$, and thus $|A|\ge
\tfrac1{d(d+1)}|E(S,P)|$.

Summing up the inequalities we obtained on $|E(S,A)|$ and $|E(S,P)|$,
we obtain $$|A|(1+\tfrac1{d+1})\ge |A|+\tfrac1{d+1}|A|\ge 
\tfrac1{d(d+1)}(|E(S,P)|+|E(S,A)|)\ge \tfrac{|S|/12}{d(d+1)},$$
and thus $|A|\ge \tfrac{|S|}{12d(d+2)}$.

 Since $G$ has average degree at most $d$, the set $R=A\cup S$
contains at least $\tfrac{n}{d+1}$ vertices and thus $$|A|\ge
\tfrac{n}{(d+1)(12d(d+2)+1)}\ge \tfrac{n}{(3d)^3},$$ where we
used that $d\ge 3$ in the last inequality. This concludes the proof of
the first part of Lemma~\ref{lem:happy}.

\medskip

Assume now that $G$ has maximum degree at most $d$. Then $P$ is empty, and so is
$E(S,P)$. It follows that $|A|\ge \tfrac1{d}|E(S,A)|\ge \tfrac{1}{12d}\,
|S|$. Since $P$ is empty, $V(G)=R=A\cup S$ and thus $|A|\ge
\tfrac{n}{12d+1}$, as desired. This concludes the proof of Lemma~\ref{lem:happy}.
\hfill $\Box$


\section{The coloring can be extended efficiently -- Proof of Lemma~\ref{lem:ext}}\label{sec:ext}

The goal of this section is to prove that any $L$-list-coloring of $G-A$ can be \emph{efficiently}
extended to $A$. (Actually our recoloring process might modify the colors of some vertices of $G\setminus A$).
To show this, we will need the notion
of an $(\alpha,\beta)$-ruling forest, introduced by Awerbuch et
al. in~\cite{AGLP89}. Given a graph $H$ and a subset $U$ of vertices
of $H$, an \emph{$(\alpha,\beta)$-ruling forest with respect to $U$}
is a family of vertex-disjoint rooted trees $(T_i,r_i)_{1\le i\le
  t}$, such that 
\begin{enumerate}
\item each vertex of $U$ lies in some tree $T_i$, and
\item for each $1\le i<j\le t$, the roots $r_i$ and $r_j$ are at
  distance at least $\alpha$ in $H$, and
\item each rooted tree $T_i$ has depth at most $\beta$ (i.e. each
  vertex of $T_i$ is at distance at most $\beta$ from $r_i$ in $T_i$).
\end{enumerate}

Awerbuch et
al. in~\cite{AGLP89} proved that a $(k,k\log n)$-ruling forest can
be computed deterministically in the \Lo\ model in a graph on at most
$n$ vertices in
$O(k \log n)$ rounds. Note that ruling forests were also used by Panconesi and
Srinivasan~\cite{PS95} in their ``distributed'' proof of Brooks
theorem, but our application here is slightly different, essentially due to the
fact that we have to consider a list-coloring problem rather than a
coloring problem (and it is seems that their approach of switching
colors cannot be applied in the list-coloring framework).

\medskip

Consider a graph $G'$, in which every vertex $v$ starts with a list
$L'(v)$ of size $d$, and assume that a subset $U$ of vertices is
precolored (i.e. each vertex of $u$ is assigned a color from its
list). Let $H$ be the subgraph of $G$ induced by
the uncolored vertices (i.e. the vertices outside $U$), and for each
$v\in H$, let $L_H(v)$
be the list obtained from $L'(v)$ by removing the colors of the
precolored neighbors of $v$. Recall that $d_H(v)$ denotes the degree
of a vertex $v$ in $H$. The following simple observation will be
crucial in our proof Lemma~\ref{lem:ext}.

\begin{obs}\label{obs:recol}
For any vertex $v\in H$, $|L_H(v)|\ge d-d_{G'}(v) +d_H(v)$. In particular, if $d_{G'}(v)\le d$ then $|L_H(v)|\ge d_H(v)$ and if $d_{G'}(v)\le d-1$ then $|L_H(v)|\ge d_H(v)+1$.
\end{obs}

\medskip

We can now proceed with the proof of Lemma~\ref{lem:ext}.

\bigskip

\noindent \emph{Proof of Lemma~\ref{lem:ext}}. 
Recall that the vertex-set $G$ is divided into two sets: $R$, the rich
vertices (that have degree at most $d$), and $P$, the poor vertices (that have degree at least $d+1$) and the set $R$ itself is divided
into $A$ (the happy vertices), and $S$ (the set of rich vertices whose
rich ball of radius $c\log n$ induces a Gallai tree and only contains
vertices of degree $d$ in $G$).

\smallskip

Consider some $L$-list-coloring of $G-A$, which we wish to extend to
$A$. 

\smallskip

Let $(T_i,r_i)_{1\le i\le
  t}$ be a $(k,k\log n)$-ruling forest in $G[R]$ with respect to $A$,
with $k=2c\log n$. This ruling
forest can be computed in $O(\log ^2 n)$ rounds as proved in~\cite{AGLP89}. Let us
denote by $T$ the union of the vertices contained in some tree
$T_i$. Note that $T$ contains $A$ (the set of vertices that are
uncolored at this point), and also possibly some (colored) vertices of $S$. We
uncolor all the vertices of $T\cap S$, and now $T$ is precisely the
set of uncolored vertices. 

Let $H$ be the subgraph of $G$
induced by $T$. For each vertex $v\in T$, we start by removing from $L(v)$ the
colors of the neighbors of $v$ outside $T$. Let us denote by $L_H(v)$ the
new list of each vertex $v$ of $H$ after this removal step. By construction, in order to extend the
coloring of $G-A$ to $A$ it is enough to
(efficiently) find an $L_H$-list-coloring of $H$.

\medskip

Since each vertex $v\in T$ has
degree at most $d$ in $G$, and starts with a list $L(v)$ of size $d$,
it follows from Observation~\ref{obs:recol} that
$|L_H(v)|\ge d_H(v)$. 
To find an $L_H$-list-coloring of $H$, we first
compute a partition of $H$ into $d+1$ stable sets
$C_1,C_2,\ldots,C_{d+1}$. Since each vertex of
$T$ has degree at most $d$ in $G$ and thus in $H$, such a partition
(which is exactly a proper $(d+1)$-coloring)
can
be computed deterministically in $O(d\log n)$
rounds~\cite{GPS88}. Recall that each tree $T_i$ has depth at most $k\log
n=2c\log^2 n$. We
then proceed to $L_H$-list-color $H$ as follows: for each $i$ from $2c\log^2 n$ to 1, and for each $j$ from 1 to $d+1$, we consider the
vertices of $C_j$ that are at distance precisely $i$ from the root of
their respective tree of the ruling forest (these vertices form a
stable set of $H$ and can thus be colored independently), and each of
these vertices selects a color of its list $L_H$ that does not appear
on any of its colored neighbors in $H$. Since each such vertex $v$ has at
least one
uncolored neighbor in $H$ (its parent in the ruling forest, since we proceed
from the leaves to the root), it has at most $d_H(v)-1\le |L_H(v)|-1$
colored neighbors and can thus select a suitable color from its list.

\medskip

This procedure takes $O(d \log^2 n)$ steps, and all that remains to do
is to find suitable colors for the roots $r_i$, $1\le i \le t$. By the
definition of a ruling forest, each root $r_i$ lies in $A$, and thus the ball $B_R(r_i)$
(of radius $c\log n$) contains a vertex of degree at most $d-1$ or is not a
Gallai tree. Moreover, it also follows from the defintion of our
ruling forest that any two balls $B_R(r_i)$  and $B_R(r_j)$ are
disjoint and have no edges between them. We then uncolor each of the
balls $B_R(r_i)$ completely, and remove from the lists of the vertices
inside these balls the colors assigned to their neighbors outside the
balls. As before, it follows from Observation~\ref{obs:recol} that the list of remaining
colors of each vertex $v$ is as least as large as the number of
uncolored neighbors of $v$. Moreover, if $v$ has degree at most $d-1$
in $G$, the list of remaining
colors for $v$ is strictly larger than the number of
uncolored neighbors of $v$. It thus follows from the definition of $A$
that each vertex $r_i$ can apply
Theorem~\ref {thm:Gallai} to its ball $B_R(r_i)$, and thus extend the current $L_H$-list-coloring to these
balls in $k\log n=2c\log^2 n=O(\log^2 n)$ additional steps. This concludes the proof of Lemma~\ref{lem:ext}.
\hfill $\Box$

\section{Conclusion}\label{sec:conclusion}

In Theorem~\ref{thm:mad}, every vertex has the same
number of colors in its list. However, as seen in
Theorem~\ref{thm:Gallai}, in the list-coloring setting, some results
can be obtained when vertices have a varying number of colors. Indeed,
many distributed $(\Delta+1)$-coloring algorithms work in the more
general $(\mbox{deg}+1)$-list-coloring setting, where each
vertex $v$ is given a list of $d(v)+1$ colors, and the complexity of
this specific problem plays an important role in the most efficient
distributed $(\Delta+1)$-coloring algorithms to date~\cite{CLP18,HSS16}.

There
are obvious obstacles to an efficient version of
Theorem~\ref{thm:Gallai} in the distributed setting: we cannot apply
it to paths in $o(n)$ rounds. Paths are not the only difficult case to
circumvent, as one could attach a clique to every vertex on a path and
face similar issues. Therefore, we need to assume that every vertex $v$ has a
list $L(v)$ of size $d(v)$, except if $d(v)\leq 2$ or the neighbors of $v$
form a clique, in which cases $v$ has a list $L(v)$ of size
$d(v)+1$. Such a list-assignment $L$ is said to be \emph{nice}. The
same algorithm and proof go through merely by replacing $d$ with the
size of the given vertex' list. Indeed, every vertex is rich, and thus
we obtain a complexity of $O(\Delta^2 \log^3 n)$ rounds.

\begin{thm}\label{thm:degchoos}
There is a deterministic distributed algorithm that given an
$n$-vertex graph $G$ of maximum degree $\Delta$, and a nice list-assignment $L$ for the vertices of $G$, finds an $L$-list-coloring of $G$ in $O(\Delta^2\log^3 n)$ rounds.
\end{thm}

 We
note that this also implies Corollary~\ref{cor:brooks} and is another way to refine the result by Panconesi and
Srinivasan~\cite{PS95} (which is faster by a factor of $\Delta \log \Delta$).

\medskip

To conclude, observe that a major difference between coloring and list-coloring in the
distributed setting is that finding \emph{some} coloring is easy
(each vertex chooses its own identifier), while finding \emph{some} list-coloring might be non
trivial. This raises the following
intriguing question, which does not seem to have been considered
before (to the best of our knowledge).

\begin{qn}\label{qn:listn}
Given an $n$-vertex graph $G$, in which each vertex has a list of $n$
available
colors, is there a simple \emph{deterministic} distributed algorithm
list-coloring $G$ in, say, $O(\log n)$ rounds?
\end{qn}

Remark that there is a simple answer to Question~\ref{qn:listn} if we
ask for a randomized algorithm instead (see for instance~\cite{BE13} for the
description of a simple
randomized distributed $(\Delta+1)$-coloring algorithm in $O(\log n)$
rounds, which can be easily adapted to work in the list-coloring
setting).

\section*{Recent developments}

Since we made our manuscript public, two papers on close topics appeared. 

\medskip

In~\cite{GHKM18},
the authors were interested in $d$-coloring non-complete graphs of maximum degree
at most $d$ (this corresponds to the setting of our Corollary~\ref{cor:brooks}). They gave a deterministic distributed algorithm
with round complexity $O(d^{1/2} \log^2 n)$ (improving our bound of $O(d^2 \log^3
n)$), and a randomized distributed
algorithm with round complexity $d^{1/2+o(1)}
(\log \log n)^2$. Using known lower bounds on the round complexity
of the deterministic version of the problem (essentially $\Omega(\log n)$), this
implies an exponential separation between the deterministic and
randomized versions of the problem.

\medskip

In~\cite{CM18}, the authors were interested specifically in planar
graphs. They gave deterministic distributed algorithms to 4-color
triangle-free planar graphs and 6-color planar graphs in $O(\log n)$
rounds (improving our round complexity of $O(\log^3 n)$ from
Corollary~\ref{cor:planar}). They also gave a different proof of
Theorem~\ref{th:lplan}.

\begin{acknowledgement} The authors would like to thank Leonid
  Barenboim and Michael Elkin
  for the discussion about the parallel 5-coloring algorithm for
  planar graphs 
  of~\cite{GPS88} mentioned in~\cite{BE10}, Jukka Suomela for providing helpful references, Cyril Gavoille for his
  comments and suggestions, and Zden\v{e}k
  Dvo\v{r}\'ak and Luke Postle for the
  discussions about coloring embedded graphs of large edge-width.
\end{acknowledgement}

\end{document}